\documentclass[final]{dmtcs-episciences}

\usepackage{wrapfig}
\usepackage[absolute]{textpos}
\usepackage[utf8]{inputenc}
\usepackage[T1]{fontenc}

\usepackage{dsfont}
\usepackage{amsthm,amssymb,amsmath}
\newtheorem{theorem}{Theorem}[section]
\newtheorem{lemma}[theorem]{Lemma}

\usepackage{enumerate}
\usepackage{url}
\usepackage{bbm}

\def\cqedsymbol{\ifmmode$\lrcorner$\else{\unskip\nobreak\hfil
\penalty50\hskip1em\null\nobreak\hfil$\lrcorner$
\parfillskip=0pt\finalhyphendemerits=0\endgraf}\fi}

\title{Two lower bounds for $p$-centered colorings}

\author{
Lo\"ic Dubois\affiliationmark{1}
\and
Gwena\"{e}l Joret\affiliationmark{2}\thanks{supported by an ARC grant from the Wallonia-Brussels Federation of Belgium and a CDR grant from the National Fund for Scientific Research (FNRS).}
\and  
Guillem Perarnau\affiliationmark{3} \\
\and Marcin Pilipczuk\affiliationmark{4}\thanks{This research is a part of a project that has received funding from the European Research Council (ERC) under the European Union's Horizon 2020 research and innovation programme
under grant agreement 714704.}
\and François Pitois\affiliationmark{1}
}

\affiliation{
  ENS Lyon, France\\
  Computer Science Department, Universit\'{e} Libre de Bruxelles, Belgium\\
  Departament de Matem\`atiques, UPC, Spain \\
  Institute of Informatics, University of Warsaw, Poland 
}

\usepackage{todonotes}
\usepackage{xspace}

\theoremstyle{plain}

\let\plainqed\qedsymbol
\newcommand{\claimqed}{$\lrcorner$}

\newcommand{\subdiv}[2]{#1^{(#2)}}
\newcommand{\twclass}{\subdiv{\mathcal{G}}{6\tw}}
\newcommand{\tw}{\mathrm{tw}}

\keywords{$p$-centered coloring, bounded expansion, polynomial expansion}
\received{2020-06-09}
\revised{2020-10-12}
\accepted{2020-10-15}
\begin{document}
\publicationdetails{22}{2020}{4}{9}{6543}

\maketitle

\begin{abstract}
Given a graph $G$ and an integer $p$, a coloring $f : V(G) \to \mathbb{N}$ is \emph{$p$-centered}
if for every connected subgraph $H$ of $G$, either $f$ uses more than $p$ colors on $H$ or there is a color that appears exactly once in $H$. 
The notion of $p$-centered colorings plays a central role in the theory of sparse graphs. 
In this note we show two lower bounds on the number of colors required in a $p$-centered coloring. 

First, we consider monotone classes of graphs whose shallow minors have average degree bounded polynomially in the radius, or equivalently (by a result of Dvo\v{r}\'ak and Norin), admitting strongly sublinear separators. We construct such a class such that $p$-centered colorings require a number of colors super-polynomial in $p$. This is in contrast with a recent result of Pilipczuk and Siebertz, who established a polynomial upper bound in the special case of graphs excluding a fixed minor. 

Second, we consider graphs of maximum degree $\Delta$. D\k{e}bski, Felsner, Micek, and Schr\"{o}der recently proved that these graphs have $p$-centered colorings with $O(\Delta^{2-1/p} p)$ colors. 
We show that there are graphs of maximum degree $\Delta$ that require $\Omega(\Delta^{2-1/p} p  \ln^{-1/p}\Delta)$ colors in any $p$-centered coloring, thus matching their upper bound up to a logarithmic factor. 
\end{abstract}

\begin{textblock}{20}(0, 12.9)
\includegraphics[width=40px]{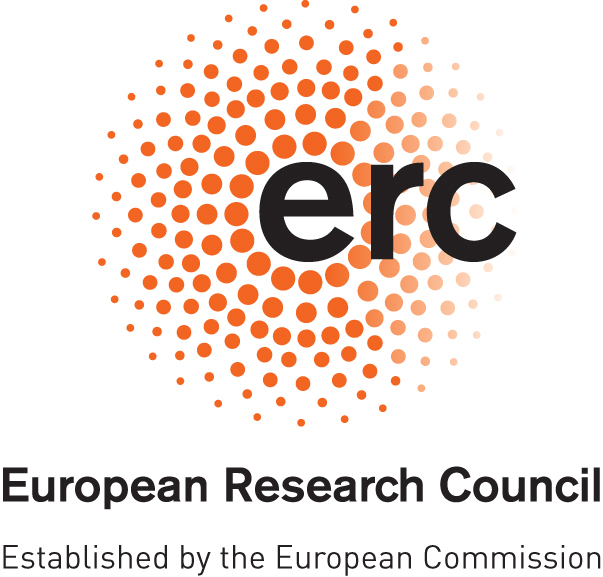}%
\end{textblock}
\begin{textblock}{20}(-0.25, 13.3)
\includegraphics[width=60px]{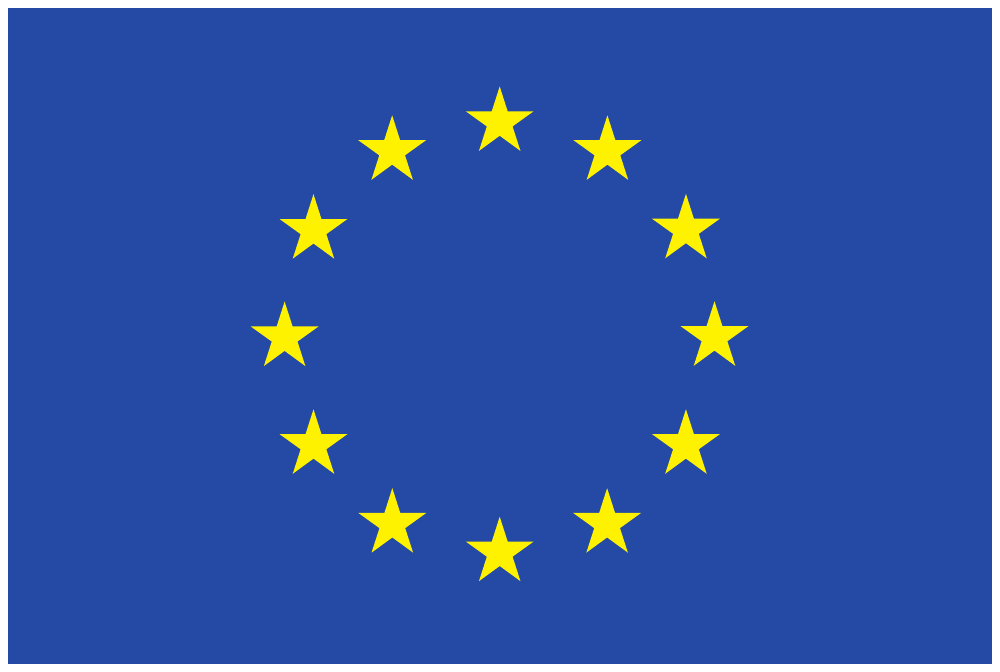}%
\end{textblock}

\section{Introduction}
Given a graph $G$ and an integer $p$, a coloring $f : V(G) \to \mathbb{N}$ is \emph{$p$-centered}
if for every connected subgraph $H$ of $G$, either $f$ uses more than $p$ colors on $H$ (i.e., $|f(V(H))| > p$) or 
there is a color that appears exactly once in $H$ (i.e., there exists $i \in \mathbb{N}$ such that $|f^{-1}(i) \cap V(H)| = 1$).

The notion of $p$-centered colorings plays a crucial role in the theory of sparse graphs. 
First, a $p$-centered coloring of a graph $G$ with small number of colors is very useful in algorithm design, for example in the task of finding or counting small subgraphs~\cite{experiments}. 
Second, a classic result shows that a graph class has bounded expansion (in the sense of Ne{\v{s}}et{\v{r}}il and Ossona de Mendez~\cite{sparsity}) if and only if for every $p \in \mathbb{N}$ there exists $M(p)$ such that every graph $G$ admits a $p$-centered coloring with at most $M(p)$ colors~\cite{NesetrilM08}. 

A recent experimental study~\cite{experiments} shows that the existing algorithms that find $p$-centered colorings find colorings with too many colors for the subsequent algorithms
to be efficient. This motivated a number of recent theoretical results trying to establish better upper and lower bounds on the number of colors needed for a $p$-centered coloring for various classes of sparse graphs~\cite{DFMS19,GroheKRSS18,PilipczukS19}. 
In this note, we contribute two lower bounds to this direction.

\paragraph{Superpolynomial lower bound in graph classes of polynomial expansion.}
A \emph{minor model} of a graph $H$ in a graph $G$ is a collection $(I_h)_{h \in V(H)}$ of vertex-disjoint connected subgraphs of $G$ such that 
$h_1h_2 \in E(H)$ implies that there is an edge of $G$ with one endpoint in $I_{h_1}$ and one endpoint in $I_{h_2}$. Fix an integer $r \geq 0$.
A minor model is \emph{$r$-shallow}
if every graph $I_h$ is of radius at most $r$. 
For a graph $G$, by $\nabla_r(G)$ we denote the maximum density of $r$-shallow minors of $G$, that is, the maximum ratio $|E(H)|/|V(H)|$ over all graphs $H$
that admit an $r$-shallow minor model in $G$. 
For a graph class $\mathcal{G}$, we denote $\nabla_r(\mathcal{G}) = \sup \{\nabla_r(G)~|~G \in \mathcal{G}\}$. 
By definition, $\mathcal{G}$ is of \emph{bounded expansion} if for every $r \geq 0$ the value $\nabla_r(\mathcal{G})$ is finite. 

The class $\mathcal{G}$ is of \emph{polynomial expansion} if there exists a polynomial $q$ such that $\nabla_r(\mathcal{G}) \leq q(r)$ for every $r \geq 0$. 
The notion of polynomial expansion turned out to be pivotal for approximation algorithms:
a subgraph-closed graph class is of polynomial expansion if and only if it admits strongly sublinear separators~\cite{DvorakN16}
and such separators allow approximation schemes via local search for a number of optimization problems~\cite{Har-PeledQ17}.
It is natural to ask what other properties one can infer about graph classes of polynomial expansion, and in particular whether they admit $p$-centered colorings with a number of colors polynomial in $p$. 
Pilipczuk and Siebertz~\cite{PilipczukS19} recently proved that such a polynomial upper bound holds in the special case of graphs excluding a fixed minor, which motivates our investigation.  

In~\cite{GroheKRSS18} an intricate example of a graph class with polynomial expansion is shown to have super-polynomial (in the radius) weak coloring numbers.
(For precise definitions, we refer to the textbook~\cite{sparsity} or the recent lecture notes~\cite{sparsitynotes}.)
In Section~\ref{sec:polyexp} we show that the same graph class requires a super-polynomial (in $p$) number of colors for a $p$-centered coloring. More precisely, we show the following.
\begin{theorem}\label{thm:lb1}
There exists a graph class $\mathcal{G}$ and a constant $c>0$ such that
$\nabla_r(\mathcal{G}) \leq r+2$ for every integer $r \geq 0$, but 
for every integer $p \geq 1$ 
there is $G \in \mathcal{G}$ such that every $p$-centered coloring of $G$
requires at least $2^{c p^{1/2}}$ colors.
\end{theorem}
The proof of Theorem~\ref{thm:lb1} builds substantially on a recent lower bound of~\cite{DFMS19}
for graphs of bounded treewidth.

We remark here that if $\mathcal{G}$ is of polynomial expansion, 
then there exists a polynomial $q$ such that for every $G \in \mathcal{G}$
and $p \geq 1$ there exists a $p$-centered coloring of $G$ with at most $2^{2^{q(p)}}$ colors.
See Chapter~2 of the lecture notes~\cite{sparsitynotes} for an exposition.

\paragraph{Lower bound for graphs of bounded degree.}
A recent breakthrough result of D\k{e}bski, Felsner, Micek, and Schr\"{o}der~\cite{DFMS19} asserts
that graphs of bounded degree require much less colors for a $p$-centered coloring 
than was anticipated.
\begin{theorem}[\cite{DFMS19}]\label{thm:bnddeg}
There exists a constant $C > 0$ such that for every integers $\Delta \geq 1$ and $p \geq 1$,
every graph $G$ of maximum degree at most $\Delta$ admits a $p$-centered coloring
with at most $C \cdot \Delta^{2-1/p} \cdot p$ colors.
\end{theorem}

In Section~\ref{sec:bnddeg} we show that the dependency on $p$ and $\Delta$ in Theorem~\ref{thm:bnddeg} 
is optimal up to a logarithmic factor in $\Delta$.
\begin{theorem}\label{thm:lb2}
There exists a constant $c > 0$ such that for every integer $p \geq 1$ there exists $\Delta_p$ such that for all $\Delta \geq \Delta_p$ there exists a graph $G$ of maximum degree at most $\Delta$ such that every $p$-centered
coloring of $G$ requires at least $c \cdot \Delta^{2-1/p} p \ln^{-1/p}{\Delta}$ colors.
\end{theorem}

We remark that this lower bound is obvious for $p=1$, and was previously known for $p=2$~\cite{fertin2004star}. 
Whether the $\ln^{-1/p}{\Delta}$ factor is necessary is an open problem. 

\section{Graphs of polynomial expansion}\label{sec:polyexp}
For a graph $G$ and an integer $t$, let $\subdiv{G}{t}$ be the graph obtained from $G$ by subdividing every edge
$t$ times, that is, replacing every edge of $G$ with path of length $t+1$ (with $t$ internal vertices). 
The newly inserted vertices are called the \emph{fresh} vertices and the vertices of $G$ are called the
\emph{root} vertices of $\subdiv{G}{t}$.

For a graph $G$, let $\tw(G)$ be the treewidth of $G$.
Grohe et al.~\cite{GroheKRSS18} observed that the class $\twclass$ consisting of the graph $\subdiv{G}{6\tw(G)}$ for
all graphs $G$ has polynomial expansion.
\begin{theorem}[\cite{GroheKRSS18}]
For every $G \in \twclass$ and integer $r \geq 0$, $\nabla_r(G) \leq r+2$.
\end{theorem}
They also observed that $\twclass$ admits only superpolynomial bounds for weak coloring numbers. 
We complete the analysis by showing that graphs in $\twclass$ require a superpolynomial number of colors (in $p$) for $p$-centered colorings.
\begin{theorem}\label{thm:lb:tw}
There exists a constant $c>0$ such that for every integer $p \geq 1$ 
there is $G \in \twclass$ such that every $p$-centered coloring of $G$
requires at least $2^{c p^{1/2}}$ colors.
\end{theorem}
This section is devoted to the proof of Theorem~\ref{thm:lb:tw}, as it immediately implies Theorem~\ref{thm:lb1}.
The construction is strongly based on a related lower bound of D\k{e}bski et al.~\cite{DFMS19}
that showed that for every integers $p,t \geq 1$ there is a graph $G_{p,t}$ of treewidth $t$
that requires $\binom{p+t}{t}$ colors in any $p$-centered coloring. 
We slightly inflate their construction and show that, after the inflation, $\subdiv{G_{p,t}}{6t}$
also requires roughly $\binom{p+t}{t}$ colors in any $p'$-centered coloring, for some $p'$ slightly larger than $p$.

The construction is parameterized by four integers $p \geq 1$, $t \geq 1$, $n_1 \geq 2$, and $n_2 \geq 2$. 
We inductively define graphs $G_{\pi,\tau}$ for $0 \leq \pi \leq p$ and $0 \leq \tau \leq t$ as follows.
In the base case, $G_{0,\tau}$ and $G_{\pi,0}$ are defined to be edgeless graphs on $2n_1^{t+1} \cdot n_2^{6t^2}$ vertices. 
For $1 \leq \pi \leq p$ and $1 \leq \tau \leq p$, the graph $G_{\pi,\tau}$ consists of a copy $G_{\pi,\tau,\bot}$
of $G_{\pi-1,\tau}$ and, for every $u \in V(G_{\pi,\tau,\bot})$, a copy $G_{\pi,\tau,u}$ of $G_{\pi,\tau-1}$ that is made fully adjacent to $u$.
The construction of $G_{\pi,\tau}$ differs from the corresponding construction of~\cite{DFMS19} only in the base case:
we choose much larger independent sets to start with. 
It is easy to see (and a formal argument can be found in~\cite{DFMS19}) that $\tw(G_{\pi,\tau}) = \tau$ when $\pi \geq 1$.
In what follows we study $p'$-centered colorings of $\subdiv{G_{p,t}}{6t}$ for some integer $p'$ slightly larger than $p$.

It will be convenient to treat root and fresh vertices separately and assign colors from disjoint palletes to them.
We say that $f$ is an $(n_1,n_2)$-coloring of $\subdiv{G_{p,t}}{6t}$ if its codomain is a union of two disjoint
sets $A_1$ and $A_2$ with $|A_1| = n_1$ and $|A_2| = n_2$ such that root vertices get assigned colors from $A_1$
and fresh vertices get assigned colors from $A_2$.
A $(n_1,n_2)$-coloring $f$ is $(p_1,p_2)$-centered if for every connected subgraph $H$ of $\subdiv{G_{p,t}}{6t}$,
either $f$ attains more than $p_1$ colors of $A_1$ on root vertices of $H$
or $f$ attains more than $p_2$ colors of $A_2$ on fresh vertices of $H$,
or there is a vertex of $H$ of unique color, that is, a color $i$ with $|f^{-1}(i) \cap V(H)| = 1$.
We prove the following statement.

\begin{lemma}\label{lem:lb:tw}
For every integers $p,t,n_2 \geq 1$, if we set $n_1 = \binom{p+t}{p}$, then
every $(3p+2, 18tp + 6t)$-centered $(n_1,n_2)$-coloring of $\subdiv{G_{p,t}}{6t}$ 
uses all $n_1$ colors of $A_1$ in its range.
\end{lemma}

To see why Lemma~\ref{lem:lb:tw} implies Theorem~\ref{thm:lb:tw}, set $t = p$ and $n_2 = n_1 = \binom{p+t}{p}$
and assume that $G := \subdiv{G_{p,p}}{6p}$ admits a $(18p^2 + 9p+2)$-centered coloring $f$ with a set $A$ of at most $\binom{p+t}{p}$ colors.
Let $A_i = \{i\} \times A$ for $i=1,2$ and let $f'(v) = (1,f(v)) \in A_1$ for a root vertex $v$ and $f'(v) = (2,f(v))$ for a fresh vertex $v$.
Then $f'$ is a $(n_1,n_2)$-coloring. Furthermore, $f'$ is $(3p+2,18tp+6t)$-centered: every connected subgraph of $G$ that uses at most $3p+2$ colors
from $A_1$ and at most $18tp+6t$ colors from $A_2$ uses at most $(18p^2+9p+2)$ colors in total and thus admits a vertex of unique color as $f$ is $(18p^2+9p+2)$-centered.
By Lemma~\ref{lem:lb:tw}, $f'$ uses all colors of $A_1$. Hence, $f$ uses at least $\binom{p+t}{p} = \binom{2p}{p} \geq 2^p$ colors. 
This finishes the proof of Theorem~\ref{thm:lb:tw}, assuming Lemma~\ref{lem:lb:tw}.

It remains to prove Lemma~\ref{lem:lb:tw}. To this end, we need a few definitions. 
Following~\cite{DFMS19}, for a color $i \in A_1$ and integers $k_1$ and $k_2$, a connected subgraph $H$ of $G$ is an \emph{$i$-threat of load $(k_1,k_2)$}
if $i$ is the only color that appears exactly once on $H$ and at most $k_1$ colors of $A_1$ and at most $k_2$ colors of $A_2$ appear on vertices of $H$.
We prove inductively the following claim.
\begin{lemma}\label{lem:lb:threats}
For every $0 \leq \pi \leq p$ and $0 \leq \tau \leq t$
and any copy $G'$ of $\subdiv{G_{\pi,\tau}}{6t}$ in $G = \subdiv{G_{p,t}}{6t}$
there exist a set $X \subseteq V(G')$ of root vertices of size $2(n_1 n_2^{6t})^{t-\tau}$
and a set $I \subseteq A_1$ of size $\binom{\pi + \tau}{\tau}$
such that for every $x \in X$ and $i \in I$ there exists an $i$-threat $H_{x,i}$ of load $(\pi + 1, 6t\pi)$ that contains $x$.
\end{lemma}
\begin{proof}
We prove the lemma by induction on $\pi + \tau$. 
For the base case, if $\pi = 0$ or $\tau = 0$, $G_{\pi,\tau}$ is an edgeless graph with $2n_1^{t+1} n_2^{6t^2}$ root vertices. 
Since $f$ is an $(n_1,n_2)$-coloring, for every copy $G'$ of $\subdiv{G_{\pi,\tau}}{6t}$ in $G$, there is a color $i \in A_1$
that is attained on a set $X$ of at least $2(n_1 n_2^{6t})^t$ vertices. Since $\{x\}$ is an $i$-threat of load $(1, 0)$ for every $x \in X$, the claim holds.

Consider now the case $\pi,\tau \geq 1$ and let $G'$ be a copy of $\subdiv{G_{\pi,\tau}}{6t}$ in $G$.
We apply the induction hypothesis to $\subdiv{G_{\pi,\tau,\bot}}{6t}$ (isomorphic to $\subdiv{G_{\pi-1,\tau}}{6t}$), obtaining a set $I_\bot$ of size $\binom{\pi-1+\tau}{\tau}$
and a set $X_\bot$ of root vertices. Pick arbitrary $x \in X_\bot$ and consider the graph $\subdiv{G_{\pi,\tau,x}}{6t}$, which is isomorphic to $\subdiv{G_{\pi,\tau-1}}{6t}$.
Applying again the induction hypothesis to this set, we obtain a set $I_x$ of size $\binom{\pi+\tau-1}{\tau-1}$ and a set $X_x$ of root vertices. 

With every $y \in X_x$ we associate a tuple $\phi(y)$ which consists of $f(y)$ and the sequence of colors of the $6t$ fresh vertices on the path between $y$ and $x$. 
Since there are $n_1 n_2^{6t}$ possible values of $\phi(y)$ and $|X_x| = 2(n_1 n_2^{6t})^{t-(\tau-1)}$, there exists a set $X \subseteq X_x$ of size $2(n_1 n_2^{6t})^{t-\tau}$
such that $\phi$ is constant on $X$. 

Assume there exists $i \in I_\bot \cap I_x$. Pick two distinct vertices $y_1,y_2 \in X$ (note that $|X| \geq 2$) and consider a subgraph $H$ of $G'$ that consists
of the $i$-threat $H_{i,x}$ of load $(\pi,6t(\pi-1))$ in $\subdiv{G_{\pi,\tau,\bot}}{6t}$,
the $i$-threat $H_{i,y_1}$ of load $(\pi+1,6t\pi)$ and 
the $i$-threat $H_{i,y_2}$ of load $(\pi+1,6t\pi)$ in $\subdiv{G_{\pi,\tau,x}}{6t}$
and the paths between $x$ and $y_1$ and $y_2$. From the loads we infer that on $H$ the coloring $f$ attains at most $3\pi+2 \leq 3p+2$ colors from $A_1$
and at most $18t\pi + 6t \leq 18tp + 6t$ colors from $A_2$. However, since $\phi(y_1) = \phi(y_2)$, $H$ has no unique color. This is a contradiction with the properties of $f$.
Hence, $I_\bot \cap I_x = \emptyset$.

Let $I = I_\bot \cup I_x$. We have $|I| = \binom{\pi-1+\tau}{\tau} + \binom{\pi+\tau-1}{\tau-1} = \binom{\pi+\tau}{\tau}$. 
It remains to show that for every $i \in I$ and $y \in X$ the graph $G'$ contains an $i$-threat of load $(\pi+1, 6t\pi)$.
This is immediate for $i \in I_x$ from the properties of $I_x$ and $X_x \supseteq X$. 
For $i \in I_\bot$, consider the $i$-threat $H_{i,x}$ of load $(\pi, 6t(\pi-1))$ in $\subdiv{G_{\pi,\tau,\bot}}{6t}$.
Define $H$ to be $H_{i,x}$ extended with a path from $x$ to $y$ and $x$ to $y'$ for some $y' \in X$, $y \neq y'$. 
Then, $H$ is connected, contains $y$, and is an $i$-threat of load $(\pi+1, 6t\pi)$ as desired. 
This finishes the proof of the lemma.
\end{proof}

By Lemma~\ref{lem:lb:threats}, in $G = \subdiv{G_{p,t}}{6t}$ we obtain a set $I \subseteq A_1$ of size $\binom{p+t}{t}$ and a set $X$ of size $2$.
The existence of the corresponding threats in $G$ ensure that $f$ attains at least $\binom{p+t}{t}$ colors on root vertices.
This finishes the proof of Lemma~\ref{lem:lb:tw} and of Theorem~\ref{thm:lb:tw}.

\section{Bounded degree graphs}\label{sec:bnddeg}
In this section we prove Theorem~\ref{thm:lb2}. The proof follows the same strategy as the current best lower bounds for acyclic and star colorings for graphs with bounded degrees~\cite{alon1991acyclic,fertin2004star}. 

Note that Theorem~\ref{thm:lb2} is obvious for $p=1$. Fix $p\geq 2$ and choose $n \in \mathbb{N}$ large enough with respect to $p$. Define
\begin{align*}
q_n := ((12e/p)^p n^{1-p} \ln n)^{\frac{1}{2p-1}}
\end{align*}
We will show that the Erd\H os-R\'enyi random graph $G(n,q_n)$ satisfies the desired lower bound with high probability. Denote by $d_n$ the maximum degree of $G(n,q_n)$. As $nq_n/\ln n \to \infty$, standard results on Random Graph Theory (see e.g.~\cite[Theorem 3.4]{frieze}) imply that 
\begin{align}\label{eq:deg}
\mathbb{P}[nq_n/2 \leq d_n \leq 2nq_n]= 1-o(1).
\end{align}
All the asymptotic notations in this section refer to $n\to \infty$. If $d_n\leq 2nq_n$, then
\begin{align*}
d_n\leq 2 \left(\frac{12e n}{p}\right)^{\frac{p}{2p-1}} (\ln n)^{\frac{1}{2p-1}}   
\end{align*}
Additionally, if $d_n\geq nq_n/2$, then for large $n$ we have $n\leq d_n^2$ and
\begin{align*}
n \geq  2^{-\frac{2p-1}{p}} \frac{p}{12 e} d_n^{2-1/p} (\ln n)^{-1/p} \geq \frac{p}{48 e} d_n^{2-1/p} (\ln d_n)^{-1/p} .
\end{align*}

We will show that, with high probability, $G(n,q_n)$ has no $p$-centered coloring using at most $n/2$ colors. Together with~\eqref{eq:deg}, this gives the existence of a graph satisfying the theorem for $c=\frac{1}{96e}$.

Consider a coloring of the vertex set of $G(n,q_n)$ using at most $n/2$ colors and let $V_1,\dots,V_{n/2}$  be the (possibly empty) color classes. For the sake of simplicity, we may assume that $n$ is multiple of $4$. We can select $m= n/4$ disjoint subsets $U_1,\dots ,U_{m}$ such that for every $ i \in [m]$ we have $U_i \subseteq V_j$ for some $j\in [n/2]$ and $|U_i| = 2$. Denote $U_i = \{x_i, y_i\}$.
Let $\mathcal{S}$ be the set of ordered subsets $\mathbf{s}=(s_1,\dots ,s_p) \subseteq [m]$ and let $\Sigma$ be the set of permutations $\sigma$ of length $p$ satisfying $\sigma(p) = 1$. To each $(\mathbf{s},\sigma)\in \mathcal{S}\times \Sigma$ we associate the following edge-set:
\begin{align*}
E_{\mathbf{s}, \sigma} = \{x_{s_i} x_{s_{i+1}} \; : \;  i\in [p-1]\}\cup \{x_{s_1} y_{s_{\sigma(1)}}\} \cup \{y_{s_{\sigma(i)}} y_{s_{\sigma(i+1)}} \; : \; i\in [p-1]\}. 
\end{align*}
The edges in $E_{\mathbf{s}, \sigma}$ span a path of length $2p-1$, with endpoints $x_{s_p}$ and $y_{s_{\sigma(p)}}$. By the choice of $\sigma$, the $p$-th and $(2p)$-th vertices in the path, belong to the set $U_{s_1}$. 

Let $E$ be the edge-set of the complete graph on $n$ vertices, and let $E_{q_n}\subseteq E$ be the random set obtained by adding each element of $E$ independently with probability $q_n$, i.e.\ the edge-set of $G(n,q_n)$. For each $(\mathbf{s}, \sigma)\in \mathcal{S}\times \Sigma$, consider the event $A_{\mathbf{s}, \sigma}= [E_{\mathbf{s}, \sigma}\subseteq E_{q_n}]$. Set 
\begin{align*}
X &= \sum_{(\mathbf{s}, \sigma)\in \mathcal{S}\times \Sigma} \mathbbm{1}_{A_{\mathbf{s}, \sigma}}\\
\mu &= \mathbb{E}[X] = \sum_{(\mathbf{s}, \sigma)\in \mathcal{S}\times \Sigma} \mathbb{P}[A_{\mathbf{s}, \sigma}]\\
\Delta &= \sum_{A_{\mathbf{s}, \sigma} \sim A_{\mathbf{s}', \sigma'}} \mathbb{P}[A_{\mathbf{s}, \sigma} \cap A_{\mathbf{s}', \sigma'}]
\end{align*}
where we write $A_{\textbf{s}, \sigma} \sim A_{\textbf{s'}, \sigma'}$ if $E_{\textbf{s}, \sigma} \cap E_{\textbf{s'}, \sigma'}\neq \emptyset$. 

Janson's inequality~(see e.g.~\cite[Theorem 8.1.1]{alon}) states that
\begin{align}\label{eq:Janson}
\mathbb{P}[X=0]\leq \exp\{-\mu+\Delta/2\}.
\end{align}

We have
\begin{align*}
\mu = (m)_p  (p-1)! q_n^{2p-1} \geq \frac{1}{2} m^p (p-1)! q_n^{2p-1} \geq \frac{1}{2p} (\frac{np}{4e})^p q_n^{2p-1} = \frac{3^p}{2p} n \ln n\geq (3/2) n\ln n,
\end{align*}
where we used that $(m)_p\geq m^p/2$ holds for $m$ sufficiently large with respect to $p$, and that $p!\geq (p/e)^p$.

For any $\mathbf{s},\mathbf{s'}\in \mathcal{S}$, we use $\mathbf{s}\cap \mathbf{s'}$ to denote the intersection of the ordered sets, as \emph{unordered} sets. If $A_{\mathbf{s}, \sigma} \sim A_{\mathbf{s}', \sigma'}$ then necessarily $|\mathbf{s} \cap \mathbf{s}'| \geq 2$. In order to bound $\Delta$ from above, we will count the contribution of all pairs of events whose sequences intersect in at least two elements. For any $2 \leq i \leq p$ we have 
\begin{align*}
|\{((\mathbf{s}, \sigma),(\mathbf{s}', \sigma'))\in (\mathcal{S}\times \Sigma)^2 :\,  |\mathbf{s} \cap \mathbf{s}'| = i\}| = \binom{m}{i} \binom{m-i}{p-i} \binom{m-p}{p-i} (p!(p-1)!)^2 = O( n^{2p-i}).
\end{align*}

We claim that if $(\mathbf{s},\sigma)\neq (\mathbf{s}',\sigma')$, then $|E_{\mathbf{s}, \sigma} \cap E_{\mathbf{s}', \sigma'}| \leq 2|\mathbf{s} \cap \mathbf{s}'| - 2$ and so $\mathbb{P}(A_{\mathbf{s}, \sigma} \cap A_{\mathbf{s}', \sigma'}) \leq q_n^{4p - 2|\mathbf{s} \cap \mathbf{s}'|}$. Clearly, we have that $|E_{\mathbf{s}, \sigma} \cap E_{\mathbf{s}', \sigma'}| \leq 2|\mathbf{s} \cap \mathbf{s}'| -1$ since the intersection spans a collection of paths. In particular, $|E_{\mathbf{s}, \sigma} \cap E_{\mathbf{s}', \sigma'}| = 2|\mathbf{s} \cap \mathbf{s}'| -1$ only if $E_{\mathbf{s}, \sigma} \cap E_{\mathbf{s}', \sigma'}$ spans a path containing the edge $x_{s_1} y_{s_{\sigma(1)}}$. As $\sigma(p)=1$, we have $s_{\sigma(1)},s_{\sigma(p)}\in \mathbf{s} \cap \mathbf{s}'$, implying that  $s_i\in \mathbf{s} \cap \mathbf{s}'$ for every $ i\in [p]$, or equivalently, $|\mathbf{s}\cap \mathbf{s}'|=p$. In such a case, the only way that $|E_{\mathbf{s}, \sigma} \cap E_{\mathbf{s}', \sigma'}|=2p-1$ is that $\mathbf{s}=\mathbf{s'}$ and $\sigma=\sigma'$.

Note that $nq_n^2\to \infty$. It follows that
\begin{align*}
\Delta = O\left(\sum_{i=2}^p n^{2p-i} q_n^{4p-2i}\right) = O( (nq_n^2)^{2p-2}) = O(n^{1-\frac{1}{2p-1}}(\ln n)^{\frac{4p-4}{2p-1}}) = o(\mu).
\end{align*}
Now we can apply~\eqref{eq:Janson} and obtain
\begin{align*}
\mathbb{P}[X = 0] &\leq \exp\{- \mu + \Delta/2\} = o(n^{-n}).
\end{align*}
Since there are at most $n^{n/2}$ colorings of the vertex set with at most $n/2$ colors, the probability that at least one of them is $p$-centered is, by a union bound, $o(1)$, concluding the proof of Theorem~\ref{thm:lb2}.

\acknowledgements
Part of this research has been done at Structural Graph Theory Downunder workshop at MATRIX research center in November 2019.
We are very thankful to MATRIX for their hospitality.
We also thank Tereza Klimo\v{s}ov\'{a}, Jan Volec, David Wood, and Lena Yuditsky for insightful discussions. 
Finally, we thank the two anonymous referees for their helpful comments on an earlier version of the paper. 

\bibliographystyle{abbrv}
\bibliography{refs}

\end{document}